\documentclass[reqno,makeidx,11pt]{amsart}
\usepackage{amssymb,amsfonts,amsmath,graphicx}
\usepackage[all]{xy}
\usepackage{extarrows}
\usepackage{hyperref}

\def\C{{\mathbb{C}}}

\def\F{{\mathbb{F}}}

\def\N{{\mathbb{N}}}
\def\0{{\mathbb{O}}}

\def\R{{\mathbb{R}}}

\def\Z{{\mathbb{Z}}}

\def\cA{{\mathcal A}}
\def\cB{{\mathcal B}}
\def\cC{{\mathcal C}}

\def\cE{{\mathcal E}}

\def\cG{{\mathcal G}}
\def\cH{{\mathcal H}}

\def\cR{{\mathcal R}}

\def \supp {\hbox{Supp}}

\newcommand{\norm}[1]{{\left\|{#1}\right\|}}

\newcommand{\abs}[1]{{\left|{#1}\right|}}
\newcommand{\scal}[1]{{\left\langle{#1}\right\rangle}}
\newcommand{\set}[1]{{\left\{{#1}\right\}}}

\newcommand{\actson}{\curvearrowright}
\newcommand{\rd}{{\,\mathrm d}}
\newcommand{\Ra}{{\Rightarrow}}

\def\fig{ \centerline{Fig. \the\count200\cGlobal\advance\count200 by 1}}
\newtheorem{thm}{Theorem}[section]

\newtheorem{lem}[thm]{Lemma}
\newtheorem{prop}[thm]{Proposition}

\newtheorem*{thm*}{Theorem}
\newtheorem*{lem*}{Lemma}
\newtheorem*{cor*}{Corollary}

\theoremstyle{definition}
\newtheorem{defn}[thm]{Definition}

\newtheorem{exs}[thm]{Examples}

\theoremstyle{remark}
\newtheorem{rem}[thm]{Remark}
\newtheorem{rems}[thm]{Remarks}

\title{Amenability, exactness and weak containment property  for groupoids }

\author{Claire Anantharaman-Delaroche}
\address{Institut Denis Poisson,
Universit\'e d'Orl\'eans, Universit\'e de Tours, \newline \indent CNRS, UMR 7013,
\newline
\indent B. P. 6759, 45067 Orl\'eans Cedex 2, France}
\email{claire.anantharaman@univ-orleans.fr}
\dedicatory{Dedicated to the memory of Eberhard Kirchberg}



\begin{document}

\begin{abstract}
From the mid-1970s, Eberhard Kirchberg undertook a remarkable extensive study of $C^*$-algebras exactness whose applications  spread out to many branches of analysis.
In this review we focus on the case of groupoid $C^*$-algebras for which the notion of exactness needs to be better understood. In particular some versions of exactness play an important role in the study of the weak containment problem (WCP), that is whether the coincidence of the full and reduced groupoid $C^*$-algebras implies the amenability of the groupoid or not.
\end{abstract}
\maketitle

\section*{Introduction}
The study of $C^*$-algebras exactness was initiated by Kirchberg as early as the mid-1970s. In the short note  \cite{Ki77Proc} he announced several results whose proofs were published along with many other major contributions in the 1990s \cite{Ki91, Ki93, Ki94}.

An unexpected link was found at the end of the 1990s between the Novikov higher signatures conjecture for a discrete group $\Gamma$ and the exactness of its reduced $C^*$-algebra $C^*_{r}(\Gamma)$. This follows from the discovery by Higson and Roe \cite{HR} that a finitely generated group $\Gamma$ has the so-called Yu's property A \cite{Yu} if and only if $\Gamma$ admits an amenable action on a compact  space $X$ (then we say that the group is amenable at infinity). Such a group satisfies the above mentioned Novikov conjecture \cite{Yu, HR, H00} and morever $C^*_{r}(\Gamma)$ is exact since it is a sub-$C^*$-algebra of the nuclear crossed product $C(X)\rtimes \Gamma$.  At the same time extensive studies of amenable actions (and amenable groupoids) \cite{AD-R} and of   reduced group $C^*$-algebras  exactness \cite{KW99}  had just appeared. This series of results was finally  crowned by the proof of the fact that if $C_{r}^*(\Gamma)$ is exact then $\Gamma$ admits an amenable action on a compact  space \cite{GK, Oza, AD02}. All this raised a renewed interest about 
various notions of exactness for locally compact groups.

It turned out to be potentially interesting to extend these notions to the case of locally compact groupoids. A first attempt was presented in \cite{AD00}. A detailed presentation was made available in \cite{AD16}. Nowadays, some  versions of groupoid exactness appear to be essential in the study of the weak containment problem (WCP).

The purpose of this paper is to describe some of the history of this subject. It is organized as follows.
 In section 1, after pointing out the fact, due to Kirchberg \cite{Ki93}, that many full group $C^*$-algebras are not exact, we focus on exactness of reduced group $C^*$-algebras, compared with two other notions of exactness, namely KW-exactness introduced by Kirchberg and Wassermann in \cite{KW99}, and amenability at infinity. The three notions are equivalent for discrete groups and we describe what is known in general for locally compact groups. 

In Section 2 we recall some facts about measured and locally compact groupoids and their operator algebras and Section 3 provides a short summary about amenable groupoids. In Section 4 we introduce the notion of amenability at infinity for a locally compact groupoid. Here, compact spaces have to be replaced by locally compact spaces that are fibred on the space of units of the groupoid in such a way that the projection is a proper map. When the groupoid $\cG$ is \'etale, it has a universal fibrewise compactification $\beta_r\cG$, called its Stone-\v Cech fibrewise compactification. Then, amenability at infinity of $\cG$ is equivalent to the fact that the canonical action of $\cG$ on $\beta_r\cG$ is amenable, and can be expressed in terms of positive type kernels, exactly as for groups.

In Section 5 we describe the relations between the various notions of exactness that are defined for \'etale groupoids. They are equivalent when we assume the inner amenability of the groupoid.  Inner amenability is well understood in the group case, in particular all discrete groups are inner amenable in our sense, but this notion remains mysterious for groupoids.  In this section we also introduce a  weak notion of exactness for groupoids, that we call inner exactness. It is automatically fulfilled for transitive groupoids and in particular for all locally compact groups, but it has proven useful in other contexts. 

From 2014 many remarkable results have been obtained in the study of the (WCP) for groupoids, which crucially involve exactness. We review them in Section 6. Finally,
in the last section we recap some open problems.

\section{Exact group $C^*$-algebras}

Unlike the functor $(\cdot)\otimes_{\hbox {max}}  A$ (maximal tensor product with the $C^*$-algebra $A$), the minimal tensor product functor $(\cdot)\otimes A$ is not necessarily exact, that is, given a short exact sequence of $C^*$-algebras $0 \rightarrow I \rightarrow B \rightarrow B/I \rightarrow 0,$
the sequence
$$0 \rightarrow I\otimes A \rightarrow B\otimes A \rightarrow (B/I) \otimes A\rightarrow 0$$
is not always exact in the middle. When the functor $(\cdot)\otimes A$ is exact, one says that the $C^*$-algebra $A$ is {\it exact}.  This notion, so named in the pioneering paper \cite{Ki77Proc}, has been the subject of Kirchberg's major contributions from the end of the 1980s. As early as 1976, Simon Wassermann showed in \cite{Wass76} that the full $C^*$-algebra
$C^*(\F_2)$ of the free group $\F_2$ on two generators is not exact. In fact, when $\Gamma$ is a finitely generated residually finite group, the full group $C^*$-algebra $C^*(\Gamma)$ is exact if and only if the group $\Gamma$ is amenable\footnote{To the author's knowledge, there is so far no example of a non-amenable discrete group $G$ such  that $C^*(G)$ is exact.}(see \cite{Ki77Proc}, and \cite[Proposition 7.1]{Ki93} for a more general result). 

This is in sharp contrast with the behaviour of the reduced group $C^*$-algebras. For instance, let $G$ be a locally compact group having a closed amenable subgroup $P$ with $G/P$ compact and let $H$ be any closed subgroup of $G$. Then the full crossed product $C^*$-algebras $C(G/P)\rtimes H$ and $C_0(H\setminus G)\rtimes P$ are Morita equivalent \cite{Rie76} and $C_0(H\setminus G)\rtimes P$ is nuclear since $P$ is amenable. It follows that $C(G/P)\rtimes H$ is nuclear and that the reduced crossed product $C^*$-algebra $C(G/P)\rtimes_r H$ is  nuclear too. But the reduced group $C^*$-algebra $C^*_{r}(H)$ embeds into $C(G/P)\rtimes_r H$ since $G/P$ is compact and therefore $C^*_{r}(H)$ is an exact $C^*$-algebra. This well-known argument applies for instance to closed subgroups of almost connected groups.

A locally compact group is said to be $C^*$-{\em  exact}\footnote{This differs from the terminology used in \cite{Ki77Proc} where a group was called $C^*$-exact if its full $C^*$-algebra was exact.} if its reduced group $C^*$-algebra is $C^*$-exact. Most familiar groups are known to be exact. The first examples of discrete groups that are not $C^*$-exact are Gromov monsters \cite{Gro}. Osajda has given other examples \cite{Osa}, and he even built residually finite groups that are not $C^*$-exact \cite{Osa18}.

Clearly, an easy way to show that the reduced group $C^*$-algebra $C^*_{r}(G)$ of a locally compact group $G$ is exact is to exhibit a continuous action of $G$ on a compact space $X$ such that the reduced crossed product $C(X)\rtimes_r G$ is nuclear. When this property is fulfilled with $G$ a discrete group, it follows that the $G$-action on $X$ is (topologically) amenable \cite[Theorem 4.5]{AD87}, \cite[Theorem 5.8]{AD02}. This notion plays an important role the study of reduced group $C^*$-algebras exactness. It will be recalled in the subsequent sections in the more general context of groupoids. It is called {\it amenability at infinity} \cite[Definition 5.2.1]{AD-R} (or boundary amenability).

When a locally compact group $G$ has an amenable action on a compact space, it has a property stronger than $C^*$-exactness, that was introduced by Kirchberg and Wassermann \cite{KW99}, and is now often called KW-exactness (see for instance \cite[Definition 5.1.9]{BO}).

\begin{defn}\label{KWexact:group} A locally compact group $G$ is KW-{\it exact} if  the functor $A \mapsto A\rtimes_r G$ is exact, that is, for every short exact sequence of $G$-$C^*$-algebras $0 \rightarrow I  \rightarrow A \rightarrow A/I \rightarrow 0$, the sequence
$$0 \rightarrow I \rtimes_r G \rightarrow A\rtimes_rG \rightarrow(A/I)\rtimes_r G \rightarrow 0$$
is also exact.
\end{defn}

The  theorem below, which presents the currently known relations between the different definitions of exactness for a locally compact group involves in particular the notion of {\em inner amenability}.  Following \cite[page 84]{Pat88}, we say that a locally compact group $G$ is {\it inner amenable} if there exists an inner invariant mean  on $L^\infty(G)$, that is, a state $m$ such that $m(sfs^{-1}) = m(f)$ for every $f\in L^\infty(G)$ and $s\in G$, where $(sfs^{-1})(y) = f(s^{-1}ys)$. This is a quite weak notion (that should deserve in fact the name of weak inner amenablity) since for instance every discrete group is inner amenable in this sense (whereas Effros \cite{Effros} excludes the trivial inner invariant mean in his definition). Note that a locally compact group $G$ is amenable  if and only if $G$ is inner amenable and $C^*_{r}(G)$ is nuclear \cite{LP}. 

The importance of inner amenability when studying the relations between properties of groups and their $C^*$-algebras has also been highlighted by Kirchberg in \cite[\S 7]{Ki93} where inner amenability is called Property (Z).

\begin{thm}\label{equivexact:group} Let $G$ be a locally compact group and consider the following conditions:
\begin{itemize}
\item[(1)] $G$ has an amenable action on a compact space;
\item[(2)] $G$ is KW-exact;
\item[(3)] $G$ is $C^*$-exact.
\end{itemize}
Then (1) $\Leftrightarrow$ (2) $\Ra$ (3) and the three conditions are equivalent when $G$ is an inner amenable group  or when $C^*_{r}(G)$ has a tracial state.
\end{thm}
 
 That (1) $\Ra$ (2) $\Ra$ (3) is easy (see for instance \cite[Theorem 7.2]{AD02}). When $G$ is a discrete group the equivalence between (2) and (3) is proved  in \cite [Theorem 5.2]{KW99} and the fact that (3) implies (1) is proved in \cite{Oza}. That  (2) implies (1) for any locally compact group is proved in \cite[Theorem 5.6]{BCL}, \cite[Proposition 2.5]{OS20}. The fact that (3) implies (1)   in the case of an inner amenable locally compact group $G$ was treated in \cite[Theorem 7.3]{AD02} where we used a property of $G$ that we called Property (W). Subsequently, it was proved in \cite{CT} that this property (W) is the same as inner amenability.  The  fact that (3) implies (2) when $C^*_{r}(G)$ has a tracial state is proved in \cite{Man}. Let us point out that this latter property is equivalent  to the existence of an open amenable normal subgroup  in $G$ as shown in \cite{KR}.  
 
 Whether  (3) implies (2) holds  for any locally compact group is still open.   Note that if KW-exactness and $C^*$-exactness are equivalent for all unimodular totally
disconnected second countable groups then they are equivalent for all locally compact second
countable groups \cite{CZ}.

\section{Background on groupoids} 
We assume that the reader is familiar with the basic definitions about groupoids. We use the terminology and the notation of \cite{AD-R}. The unit space of a groupoid $\cG$ is denoted by $\cG^{(0)}$ and is often renamed as $X$. We implicitly identify $\cG^{(0)}$ to a subset of $\cG$. The structure of $\cG$ is defined by the range and source maps $r,s: \cG\to\cG^{(0)}$, the inverse map $\gamma \mapsto \gamma^{-1}$ from $\cG$ to $\cG$ and the multiplication map $(\gamma,\gamma') \mapsto \gamma\gamma'$ from 
$\cG^{(2)} = \set{(\gamma,\gamma')\in \cG\times \cG : s(\gamma) = r(\gamma')}$ to $\cG$.  For $x\in \cG^{(0)}$ we set $\cG^x = r^{-1}(x)$,  $\cG_x = s^{-1}(x)$ and $\cG(x) = \cG^x \cap \cG_x$.  Given $E\subset \cG^{(0)}$, we write $\cG(E) = r^{-1}(E)\cap s^{-1}(E)$. 

One important example is given by the left action of a group $G$ on a set $X$. The corresponding {\em semidirect product groupoid} $\cG = X\rtimes G$ is $X\times G$ as a set. Its unit set is $X$, the range and source maps are given respectively by $r(x,g) = x$ and $s(x,g)= g^{-1}x$. The product is given by $(x,g)(g^{-1}x,h) = (x,gh)$, and the inverse by $(x,g)^{-1} = (g^{-1}x, g^{-1})$. Equivalence relations on $X$ are also an interesting family of examples. If $\cR\subset X\times X$ is an equivalence relation, it is viewed as a groupoid with $X$ as set of units, $r(x,y) = x$ and $s(x,y) = y$ as range and source maps respectively. The product is given by $(x,y)(y,z) = (x,z)$ and the inverse by $(x,y)^{-1} = (y,x)$.

\subsection{Measured groupoids} A {\em Borel groupoid} $\cG$ is a groupoid endowed with a Borel structure such that the range, source, inverse and product maps are Borel, where $\cG^{(2)}$ has the Borel structure induced by $\cG\times \cG$ and $\cG^{(0)}$ has the Borel structure induced by $\cG^{(0)}$. A {\em Borel Haar system} $\lambda$ on $\cG$ is a family $(\lambda^x)_{x\in \cG^{(0)}}$ of measures on the fibres $\cG^x$, wich is Borel (in the sense that for every non-negative Borel function $f$ on $\cG$ the function $x\mapsto \lambda(f)(x) =\int f\rd\lambda^x$ is Borel), left invariant  (in the sense that for all $\gamma\in\cG, \quad\gamma\lambda^{s(\gamma)} = \lambda^{r(\gamma)}$), proper (in the sense that there exists a non-negative Borel function $f$ on $\cG$ such that $ \lambda(f)(x) = 1$ for all $x\in \cG^{(0)}$). Given a measure $\mu$ on $\cG^{(0)}$, one can integrate  the measures $\lambda^x$ with respect to $\mu$ to get a measure $\mu\circ\lambda$ on $\cG$. The measure $\mu$ is quasi-invariant with respect to the Haar system if the inverse map preserves the $(\mu\circ\lambda)$-negligible sets. A {\em measured groupoid} is a triple $(\cG,\lambda,\mu)$ satisfying the above properties. All measure spaces are assumed to be standard and the measures are $\sigma$-finite. 

\begin{exs}\label{exs:mesgroupoid} (a) {\it Semidirect product measured groupoids.} Let $G$ be a second countable locally compact group with a left Haar measure $\lambda$, and $X$ a standard Borel space.  A Borel left action of $G$ on $X$ is a left action such that the map $(x,s)\mapsto sx$ from $X\times G$ to $X$ is Borel. Then $\cG=X\rtimes G$ is a Borel groupoid with a canonical Haar system, also denoted by $\lambda$. Indeed, identifying $\cG^x$ with $G$, we take $\lambda^x = \lambda$. Let $\mu$ be a measure on $X$. Then $\mu\circ\lambda = \mu\otimes\lambda$.  Moreover $\mu$ is quasi-invariant with respect to the $G$-action if and only if $(\cG,\lambda,\mu)$ is a measured groupoid. 

(b) {\em Discrete measured equivalence relations.} Let $\cR$ be an equivalence relation on a Borel standard space $X$ which has countable equivalence classes and such that $\cR$ is a Borel subset of $X\times X$.
This groupoid has a canonical Haar system: $\lambda^x$ is the counting measure on the equivalence class of $x$, identified with $\cR^x$. A measure $\mu$ on $X$ is quasi-invariant if for every Borel subset $A\subset X$, the saturation of $A$ with respect to $\cR$ has measure $0$ when $\mu(A) = 0$. Then $(\cR,\mu)$ is a measured groupoid,  called a {\em discrete measured equivalence relation}.
\end{exs}

\subsection{Topological groupoids}
A {\it locally compact groupoid} is a groupoid $\cG$ equipped with a  locally compact\footnote{By convention a locally compact space will be Hausdorff.} topology such  that the structure maps are continuous, where $\cG^{(2)}$ has the topology induced by $\cG\times\cG$ and $\cG^{(0)}$ has the topology induced by $\cG$. A {\em continuous Haar system} is a family $ \lambda=(\lambda^x)_{x\in \cG^{(0)}}$ of measures on $\cG$ such that  $\lambda^x$ has exactly $\cG^x$ as support for every $x\in \cG^{(0)}$, is left invariant and is continuous in the sense that for every $f\in \cC_c(\cG)$ (the space of  continuous complex valued functions with compact support on $\cG$) the function $x\mapsto  \lambda(f)(x)= \int f\rd\lambda^x$ is continuous. Note that the existence of a continuous Haar system implies that the range (and therefore the source) map is open \cite[Chap. I, Proposition 2.4]{Ren_book}.

\begin{exs}\label{exs:groupoids} (a) {\it Semidirect products.} Let us consider   a locally compact group $G$ with Haar measure $\lambda$ acting continuously to the left on a locally compact space $X$. Then $\cG = X\rtimes G$ is locally compact groupoid and the Haar system defined in Example \ref{exs:mesgroupoid}  (a) is continuous.

(b) {\it  Group bundle groupoids.} A group bundle groupoid is a locally compact groupoid such that the range and source maps are equal and open. By \cite[Lemma 1.3]{Ren91}, one can choose, for $x\in \cG^{(0)}$, a left Haar measure $\lambda^x$ on the group $\cG^x =\cG_x$ in such a way that $(\lambda^x)_{x\in X}$ forms a Haar system on $\cG$. An explicit example will be given in Section \ref{sec:HLS}.

(c) {\it \'Etale groupoids.} A locally compact groupoid is called {\it \'etale} when its range (and therefore its source) map is a local homeomorphism from $\cG$ into $\cG^{(0)}$. Then $\cG^x$ and $\cG_x$ are discrete and $\cG^{(0)}$ is open in $\cG$. Moreover the family of counting measures $\lambda^x$ on $\cG^x$ forms a Haar system (see \cite[Chap. I, Proposition  2.8]{Ren_book}). It will be implicitly our choice of Haar system. Groupoids associated with actions of discrete groups are \'etale.
\end{exs}

\subsection{Groupoid operator algebras}  For the representation theory of measured groupoids we refer to \cite[\S 6.1]{AD-R}. The von Neumann algebra $VN(\cG,\lambda,\mu)$ associated to such a groupoid is defined by its left regular representation. For a semidirect product $(X\rtimes G, \mu)$ it is the von Neumann crossed product  $L^\infty(X,\mu)\rtimes G$. For a discrete measured equivalence relation $(\cR,\mu)$, it is the von Neumann algebra defined in \cite{FMII}.

We will now focus on the operator algebras associated with a locally compact groupoid\footnote{Throughout this text a locally compact groupoid will be implicitly  endowed with a Haar system $\lambda$ which, concerning the examples given in Examples \ref{exs:groupoids}, will be the Haar systems described  there.} $\cG$. We set $X = \cG^{(0)}$. The space $\cC_c(\cG)$ of continuous functions with compact support on $\cG$ is an involutive algebra with respect to the following operations for $f,g\in \cC_c(\cG)$:
\begin{align*}
(f*g)(\gamma) &= \int f(\gamma_1)g(\gamma_{1}^{-1}\gamma) d\lambda^{r(\gamma)}(\gamma_1)\\
f^*(\gamma) & =\overline{f(\gamma^{-1})}. 
\end{align*}

We define a norm on $\cC_c(\cG)$ by
$$\norm{f}_I = \max \set{\sup_{x\in X} \int \abs{f(\gamma)}\rd\lambda^x(\gamma), \,\,\sup_{x\in X} \int \abs{f(\gamma^{-1})}\rd\lambda^x(\gamma)}.$$

The {\it full $C^*$-algebra $C^*(\cG)$ of the groupoid} $\cG$ is the enveloping $C^*$-algebra of the Banach $*$-algebra obtained by completion of $\cC_c(\cG)$ with respect to the norm $\norm{\cdot}_I$.

In order to define the reduced $C^*$-algebra of $\cG$ we need the notion of (right) Hilbert   $C^*$-module $\cH$ over a $C^*$-algebra $A$ (or Hilbert $A$-module) for which we refer to \cite{Lance_book}. We shall denote by  $\cB_A(\cH)$ the $C^*$-algebra of $A$-linear adjointable maps from $\cH$ into itself.

Let $\cE$ be the Hilbert $C^*$-module\footnote{When $\cG$ is \'etale, we shall use the notation $\ell^2_{\cC_0(X)}(\cG)$ rather than  $L^2_{\cC_0(X)}(\cG,\lambda)$ } $L^2_{\cC_0(X)}(\cG,\lambda)$ over  ${\mathcal C}_0(X)$ (the algebra of continuous functions on $X$ vanishing to $0$ at infinity)  obtained by completion of $\cC_c(\cG)$ with respect to the $\cC_0(X)$-valued inner product
$$\scal{\xi,\eta}(x) = \int_{\cG^x} \overline{\xi(\gamma)}\eta(\gamma)\rd\lambda^x(\gamma).$$
The $\cC_0(X)$-module structure is given by
$$(\xi f)(\gamma) = \xi(\gamma)f\circ r(\gamma).$$
Let us observe that $L^2_{\cC_0(X)}(\cG,\lambda)$ is the space of continuous sections vanishing at infinity of a continuous field of Hilbert spaces with fibre $L^2(\cG^x,\lambda^x)$ at $x\in X$.

We let $\cC_c(\cG)$ act on $\cE$ by the formula
$$(\Lambda(f)\xi)(\gamma) = \int f(\gamma^{-1}\gamma_1) \xi(\gamma_1) \rd\lambda^{r(\gamma)}(\gamma_1).$$
Then, $\Lambda$ extends to a representation of $C^*(\cG)$ in the Hilbert $\cC_0(X)$-module $\cE$, called the {\it regular representation of} $(\cG,\lambda)$. Its range is denoted by $C^*_{r}(\cG)$ and called the {\it reduced $C^*$-algebra}\footnote{Very often, the Hilbert $\cC_0(X)$-module $L^2_{\cC_0(X)}(\cG,\lambda^{-1})$ is considered in order to define the reduced $C^*$-algebra (see for instance \cite{KS02,KS04}). We pass from this setting to ours (which we think to be more convenient for our purpose) by considering the isomorphism $U: L^2_{\cC_0(X)}(\cG,\lambda^{-1})\to L^2_{\cC_0(X)}(\cG,\lambda)$ such that $(U\xi)(\gamma) = \xi(\gamma^{-1})$.}{\it of the groupoid} $\cG$. Note that $\Lambda(C^*(\cG))$ acts fibrewise on the corresponding continuous field of Hilbert spaces with fibres $L^2(\cG^x,\lambda^x)$ by the formula
$$(\Lambda_x(f)\xi)(\gamma) = \int_{\cG^x} f(\gamma^{-1}\gamma_1) \xi(\gamma_1) \rd\lambda^x(\gamma_1)$$
for $f\in \cC_c(\cG)$ and $\xi\in L^2(\cG^x,\lambda^x)$. Moreover, we have
$\norm{\Lambda(f)} = \sup_{x\in X} \norm{\Lambda_x(f)}$.

For a semidirect product groupoid $\cG = X\rtimes G$ as in Example \ref{exs:groupoids} (a)  we get the usual crossed products $C^*(\cG) = C_0(X)\rtimes G$ and $C^*_{r}(\cG) = C_0(X)\rtimes_r G$. 

\section{Amenable groupoids}

\subsection{Amenability of measured groupoids} The existence of actions of non-amenable groups exhibiting behaviours reminiscent of amenability had already  been observed in the 1970s by several authors, among them Vershik \cite{Ver} for the boundary action of $PSL(2,\Z)$. The original definition of an amenable action in the measured setting is due to Zimmer \cite[Definition 1.4]{Zi3}. It was expressed in terms of an involved  fixed point property. Later \cite{Zi1} it was reformulated in terms of invariant means: an action of a discrete group $\Gamma$ on a measured space $(X,\mu)$, with $\mu$ being quasi-invariant, is amenable if there exists a norm one projection  $m: L^\infty(X\rtimes \Gamma, \mu\circ\lambda)\to L^\infty(X,\mu)$ such that $s.(m(f)) = m(s.f)$ for all $f\in L^\infty(X\rtimes \Gamma, \mu\circ\lambda)$ and $s\in \Gamma$ where $(s.f)(x,t) = f(s^{-1}x,s^{-1}t)$ and $(s.m(f))(x) = m(f)(s^{-1}x)$. In \cite{AEG}, this characterization was extended to the case of any second countable locally compact group. It also holds in the case of discrete measured equivalence relations. Clearly, the right framework that unifies this notion of amenability is that of measured groupoids.

\begin{defn}\label{amen:measgroupoid}\cite[Definition 3.2.8]{AD-R} A measured groupoid $(\cG,\lambda,\mu)$ is said to be {\it amenable} if there exists a norm one projection $m: L^\infty(\cG,\mu\circ\lambda) \to L^\infty(\cG^{(0)},\mu)$ such that  $m(\psi*f) = \psi* m(f)$  for every $f\in  L^\infty(\cG,\mu\circ\lambda)$ and every Borel function $\psi$ on $\cG$ such that $\sup_{x\in \cG^{(0)}}\lambda^x(\abs{\psi} )<\infty$.
\end{defn}

Recall that $(\psi*f)(\gamma) = \int \psi(\eta)f(\eta^{-1}\gamma) \rd \lambda^{r(\gamma)}(\eta)$ for $f\in  L^\infty(\cG,\mu\circ\lambda)$ and that we have $(\psi*f)(x) = \int \psi(\eta)f(\eta^{-1}x) \rd \lambda^{x}(\eta)$ for $f\in L^\infty(X,\mu)$.

The first definition of amenability for a measured groupoid is due to Renault \cite[Chap. II, \S 3]{Ren_book}. It was expressed  in different terms:  as a generalisation of the classical Day condition or equivalently  as generalisations of the Reiter condition or of the Godement condition for groups.

\begin{thm}\label{thm:caractAmenMeas}\cite[Propostion 3.2.14]{AD-R} Let $(\cG,\lambda,\mu)$ be a measured groupoid. We endow $\cG$ with the measure $\mu\circ\lambda$. The following conditions are equivalent:
\begin{itemize}
\item[(i)] $(\cG,\lambda,\mu)$ is amenable;
\item[(ii)] $[$Weak Day condition$]$ There exists a sequence $(g_n)$ of non-negative Borel functions on $\cG$ such that $\lambda(g_n) =1$ and  $\lim_n f*g_n - (\lambda(f)\circ r )g_n = 0$ in the weak topology of $L^1(\cG)$ for all  $f\in L^1(\cG)$;
 \item[(iii)] $[$Weak Reiter condition$]$ There exists a sequence $(g_n)$ of non-negative Borel functions on $\cG$ such that $\lambda(g_n) =1$ and $\lim_n \int\abs{g_n(\gamma^{-1}\gamma_1) - g_n(\gamma_1)}\rd \lambda^{r(\gamma)}(\gamma_1) = 0$ in the weak*-topology of $L^\infty(\cG)$;
 \item[(iv)] $[$Weak Godement condition$]$  There exists a sequence $(\xi_n)$ of Borel functions on $\cG$ such that $\lambda(\abs{\xi_n}^2) = 1$ for all $n$ and $\lim_n \int\overline{\xi_n(\gamma_1)}\xi_n(\gamma^{-1}\gamma_1) \rd\lambda^{r(\gamma)}(\gamma_1) = 1$ in the weak*-topology of $L^\infty(\cG)$.
\end{itemize}
\end{thm}

\subsection{Amenability of locally compact groupoids} The (topological) amenability\footnote{From now on, amenability will implicitly mean topological amenability.} of a locally compact groupoid $\cG$ has been introduced by Renault in \cite{Ren_book}.  In \cite{AD-R} it is defined as follows.

\begin{defn} \cite[Definition 2.2.1]{AD-R} We say that a locally compact groupoid $\cG$ is {\em amenable} if there exists a net (or a sequence when $\cG$ is $\sigma$-compact) $(m_i)$, where $m_i =  (m_i^{x})_{x\in \cG^{(0)}}$ is a family of probability measures $m_i^{x}$ on $\cG^x$, continuous in the sense that $x\mapsto m_i^{x}(f)$ is continuous for every $f\in \cC_c(\cG)$, and such that $\lim_i\norm{\gamma m_i^{s(\gamma)} - m_i^{r(\gamma)}}_1 = 0$ uniformly on every compact subset of $\cG$.
\end{defn}

This notion has  many equivalent definitions:

\begin{thm}\label{thm:caractAmen}\cite[Proposition 2.2.13]{AD-R} Let $\cG$ be a $\sigma$-compact locally compact groupoid. The following conditions are equivalent:
\begin{itemize}
\item[(i)] $\cG$ is amenable;
\item[(ii)] $[$Reiter condition$]$ There exists a sequence $(g_n)$ in $\cC_c(\cG)^+$ such that $\lim_n\lambda(g_n) = 1$ uniformly on every compact subset of $\cG^{(0)}$ and $\lim_n \int\abs{g_n(\gamma^{-1}\gamma_1) - g_n(\gamma_1)}\rd \lambda^{r(\gamma)}(\gamma_1) = 0$ uniformly on every compact subset of $\cG$;
\item[(iii)] There exists a sequence $(h_n)$ of continuous positive definite functions with compact support on $\cG$ whose restrictions to $\cG^{(0)}$ are bounded by $1$ and such that $\lim_n h_n = 1$ uniformly on every compact subset  of $\cG$;
\item[(iv)] $[$Godement condition$]$ There exists a sequence $(\xi_n)$  in $\cC_c(\cG)$ such that $\lambda(\abs{\xi_n}^2) \leq 1$ for all $n$ and $\lim_n \int\overline{\xi_n(\gamma_1)}\xi_n(\gamma^{-1}\gamma_1) \rd\lambda^{r(\gamma)}(\gamma_1) = 1$ uniformly on every compact subset of $\cG$.
 \end{itemize}
 \end{thm}

 Recall that a function $h$ on $\cG$ is {\em positive definite} or of {\em positive type}  if for every $x\in \cG^{(0)}$, $n\in \N$ and $\gamma_1,\cdots, \gamma_n \in \cG^x$, the $n \times n$ matrix $[h(\gamma_i^{-1}\gamma_j)]$ is non-negative. For instance, given $\xi$  on $\cG$ such that $\lambda(\abs{\xi}^2) $ is bounded on $\cG^{(0)}$, the function $\gamma \mapsto \int\overline{\xi(\gamma_1)}\xi(\gamma^{-1}\gamma_1) \rd\lambda^{r(\gamma)}(\gamma_1)$ is positive definite.
 
 \begin{rems}\label{rem:Godement} (a)  In \cite{AD-R} it is assumed that $\cG$ is second countable but the proof of the above theorem holds as well when $\cG$ is $\sigma$-compact. This observation will be useful later when working with the groupoid $\beta_r\cG \rtimes \cG$.
 
(b) In the above characterizations, the boundedness conditions for the sequences $(h_n)$ and $(\xi_n)$ are not necessary (see \cite[Propositions  2.2.13]{AD-R}).
   \end{rems}

 \begin{defn}\cite[Chap. II, Definition 3.6]{Ren_book},\cite[Definition 3.3.1]{AD-R} One says that  a  second countable locally compact groupoid with Haar system $(\cG,\lambda)$  is {\em measurewise amenable} if for every quasi-invariant measure $\mu$ on $\cG^{(0)}$ the measured groupoid $(\cG,\lambda,\mu)$ is amenable.
 \end{defn}
  
    Topological amenability is closely related to measurewise amenability.  It is not hard to see for instance that the Reiter condition of Theorem \ref{thm:caractAmen} implies the weak Reiter condition of Theorem \ref{thm:caractAmenMeas} for every quasi-invariant measure $\mu$. Therefore topological amenability implies measurewise amenability. It is a long-standing open question whether the converse is true. This has been proved for \'etale groupoids \cite[Corollary 3.3.8]{AD-R} and recently for locally compact second countable semidirect product groupoids \cite[Corollary 3.29]{BEW20}.

   \begin{rem} Let us consider the case of a  locally compact  semidirect product groupoid\footnote{For these groupoids the $\sigma$-compactness assumption is not needed \cite[Proposition 2.5]{AD02}.} $\cG= X\rtimes G$. Then, topological amenability is for instance spelled out as the existence of a net $(m_i)$ of weak*-continuous maps $m_i: x\mapsto m_i^{x}$ from $X$ into the space of probability measures on $G$, such that $\lim_i \norm{gm_i^{x} - m_i^{gx}}_1 = 0$ uniformly on every compact subset of $X\times G$. In this case we also say that {\em the $G$-action on $X$ is amenable}.
   
We set $A = \cC_0(X)$ and for every $f\in \cC_c(X\times G)$ we set $\tilde f(s)(x) = f(x,s)$. Then $\tilde f$ is in the space $\cC_c(G,A)$ of continuous functions  with compact support from $G$ into $A$. It is also an element of the Hilbert $A$-module $L^2(G,A)$ given as the completion of $\cC_c(G,A)$ with respect to the $A$-valued inner product $\scal{\xi,\eta} = \int_G \xi(s)^*\eta(s) \rd\lambda(s)$. Finally, for $\xi\in \cC_c(G,A)$, we set $(\tilde\alpha_t\xi)(s)(x) = \xi(t^{-1}s)(t^{-1}x)$ and we denote by the same symbol the continuous extension of $\tilde\alpha_t$ to $L^2(G,A)$. If $h_i(\gamma)=  \int \overline{\xi_i(\gamma_1)}\xi_i(\gamma^{-1}\gamma_1) \rd\lambda^{r(\gamma)}(\gamma_1)$ with $\xi_i\in \cC_c(X\times G)$,  we have
  $$\tilde h_i(t)(x) = \int_G \overline{\xi_i(x,s)} \xi_i(t^{-1}x,t^{-1}s)\rd\lambda(s) = \scal{\tilde\xi_i,\tilde\alpha_t(\tilde\xi_i)}(x).$$
 It follows that the Godement condition characterizing the amenability of $X\rtimes G$ may be interpreted as the existence of a bounded net $(\eta_i)$ in $L^2(G,A)$ such that $\scal{\eta_i, \tilde\alpha_t(\eta_i)} \to 1$ uniformly on compact subsets of $G$ in the strict topology of $A$.
 
 The first tentative to define  an amenable action of a group on a non-commutative $C^*$-algebra $A$ was presented in \cite{AD87}. The solution was not  satisfactory since it was limited to discrete groups and involved the bidual of $A$.  Since the end of the 2010s, a new interest in the subject has led to major advances \cite{BC, BEW, BEW20, OS20} and resulted in very nice equivalent definitions of amenability. One of the definitions is the following extension of the commutative case described above, called the approximation property (AP), first introduced in \cite{Exe, ENg} in the setting of Fell bundles over locally compact groups. An action $\alpha: G\actson A$ of a locally compact group on a $C^*$-algebra $A$ has the {\em approximation property} (AP) if there exists a bounded net (or sequence in separable cases) $(\eta_i)$ in $\cC_c(G,A) \subset L^2(G,A)$  such that $\scal{\eta_i, a \tilde\alpha_t(\eta_i)} \to a$ in norm, uniformly on compact subsets of $G$, for every $a\in A$. Here one sets again $\tilde\alpha_t(\eta)(s) = \alpha_t(\eta(t^{-1}s))$ for $\eta\in \cC_c(G,A)$, $s,t\in G$. For interesting properties of amenable actions of locally compact groups on $C^*$-algebras we refer to \cite{BEW20, OS20}.
\end{rem}

 \section{Amenable at infinity groupoids}
 As already said, the property for a locally compact group $G$ to be KW-exact is equivalent to the existence of an amenable $G$-action on a compact space. In order to try to extend this fact to a locally compact groupoid $\cG$ we need some preparation.
 
 \subsection{First definitions} Let $X$ be a locally compact space. A {\it fibre space} over $X$ is a pair $(Y,p)$ where $Y$ is a locally compact space  and $p$ is  a continuous surjective map from  $Y$ on $X$. For $x\in X$ we denote by $Y^x$ the {\it fibre} $p^{-1}(x)$.  We say that $(Y,p)$ is {\em fibrewise compact} if the map $p$ is proper in the sense that $p^{-1}(K)$ is compact for every subset $K$ of $X$. Note that this property is stronger than requiring each fibre to be compact.

Let $ (Y_i,p_i)$, $i=1,2$, be two fibre spaces over $X$. We denote by $Y_{1}   
 \,_{p_1}\!\!*_{p_2} Y_2$ (or $Y_1*Y_2$ when there is no ambiguity) the {\it fibred product}\index{fibred product} 
 $\set{(y_1,y_2)\in Y_1\times Y_2: p_1(y_1) = p_2(y_2)}$
  equipped  with the topology induced by the product topology.  We say that a continuous map $\varphi: Y_1\to Y_2$  is a {\it morphism of fibre spaces} if $p_2\circ \varphi = p_1$. 
  
 \begin{defn}\label{def:Gspace} Let $\cG$ be a locally compact groupoid. A {\it left} $\cG$-{\it space} \index{$\cG$-space} is a fibre space $(Y,p)$ over $X = \cG^{(0)}$, equipped with a continuous map $(\gamma, y) \mapsto \gamma y$ from $\cG\,_s\!*_pY$ into $Y$, satisfying the following conditions:
 \begin{itemize}
 \item $p(\gamma y) = r(\gamma)$ for $(\gamma, y) \in \cG\,_s\!*_pY$, and $p(y)y = y$ for $y\in Y$;
 \item if $(\gamma_1, y) \in \cG\,_s\!*_pY$ and $(\gamma_2,\gamma_1)\in \cG^{(2)}$, then $(\gamma_2\gamma_1)y = \gamma_2(\gamma_1 y)$.
 \end{itemize}
 \end{defn}

Given such a $\cG$-space $(Y,p)$, we associate a groupoid $Y\rtimes \cG$, called the semidirect product groupoid of $Y$ by $\cG$. It is defined as in the case of group actions except that as a topological space it is the fibred product $Y\!_p\!*_r \cG$ over $X = \cG^{(0)}$. Although $p$ is not assumed to be an open map, the range map $(y,\gamma) \mapsto y$ from  $Y\rtimes \cG$ onto $Y$ is open since the range map $r:\gamma \mapsto r(\gamma)$ is open. Moreover, if $\cG$ has a Haar system $(\lambda^x)_{x\in X}$, then  $Y\rtimes \cG$ has the canonical Haar system $y\mapsto \delta_y\times \lambda^{p(y)}$ (identified with $\lambda^{p(y)}$ on $\cG^{(p(y)}$) (see \cite[Proposition 1.4]{AD16}). Note that $Y\rtimes \cG$ is an \'etale groupoid when $\cG$ is \'etale.  

We say that the  {\em $\cG$-space $(Y,p)$ is amenable} if the semidirect product groupoid $Y\rtimes \cG$ is amenable. Note that if $\cG$ is an amenable groupoid, every $\cG$-space is amenable \cite[Corollary 2.2.10]{AD-R}.

There is a subtlety about the definition of amenability at infinity which leads us to introduce two notions. We do not know whether they are equivalent in general.

\begin{defn}\label{def:ameninf} Let $\cG$ be a locally compact groupoid and let $X= \cG^{(0)}$. We say that 
\begin{itemize}
\item[(i)] $\cG$ is {\em strongly amenable at infinity}  if there exists an amenable fibrewise compact $\cG$-space $(Y,p)$ with a continuous section $\sigma : X\to Y$ of $p$;
\item[(ii)]  $\cG$ is {\em amenable at infinity}  if there exists an amenable fibrewise compact $\cG$-space;
\end{itemize}
\end{defn}

\begin{exs}\label{rem:invGE} (a) Every locally compact amenable groupoid $\cG$ is strongly amenable at infinity since the left action of $\cG$ on its unit space is amenable.

(b)  It is easily seen that  the semidirect product groupoid $\cG = X\rtimes G$ relative to an action of  a KW-exact (hence amenable  at infinity) locally compact group $G$  on a locally compact space $X$ is  strongly amenable at infinity \cite[Proposition 4.3]{AD16}. This is also true for partial actions\footnote{For the definition see \cite[\S I.2, \S I.5]{Exel_Book}.} of exact discrete groups \cite[Proposition 4.23]{AD16}.
\end{exs}

It is useful to have a criterion  of amenablity at infinity which does not involve $Y$ but only $\cG$. Before proceeding further we need to introduce some notation and definitions. We set $\cG*_r \cG = \set{(\gamma,\gamma_1) \in \cG\times \cG : r(\gamma) = r(\gamma_1)}$. A subset of $\cG*_r \cG$ will be called a {\it tube} if its image by the map
$(\gamma,\gamma_1) \mapsto \gamma^{-1}\gamma_1$ is relatively compact in ${\mathcal G}$. We denote by
$\cC_t(\cG*_r\cG)$ the space of continuous bounded functions on $\cG*_r\cG$ with support in a tube.

We say that a function $k: \cG *_r \cG \to \C$ is a {\it positive definite kernel} if for every $x\in X$, $n \in \N$ and $\gamma_1,\dots,\gamma_n \in \cG^x$, the matrix $[k(\gamma_i,\gamma_j)]$ is non-negative, that is
$$\sum_{i,j=1}^n \overline{\alpha_i}\alpha_j k(\gamma_i,\gamma_j) \geq 0$$
for $\alpha_1,\dots,\alpha_n \in \C$.

In the case of groups (for which amenability at infinity coincides with strong amenability at infinity) let us recall the following result:

\begin{thm}\label{caractinfgroup} A (second countable) locally compact group $G$ is amenable at infinity if and only if there exists a net $(k_i)$ of continuous positive definite kernels $k_i: G\times G \to \C$ with support in tubes such that  $\lim_i k_i = 1$ uniformly  on tubes. 
\end{thm}

When $G$ is any discrete group  this is proved in \cite{Oza} and when $G$ is a locally compact second countable group this is proved in \cite[Theorem 2.3, Corollary 2.9]{DL} which improves \cite[Proposition 3.5]{AD02}. One important ingredient in the proof of the above theorem is the use of a universal compact $G$-space, namely the Stone-\v Cech compactification $\beta G$ of $G$ if $G$ is discrete and an appropriate variant of it in general.

\subsection{Fibrewise compactifications of $\cG$-spaces} In order to extend Theorem \ref{caractinfgroup}  to the case of groupoids we first need some informations about fibrewise compactifications of fibre spaces.

\begin{defn}\label{def:fibComp}
 A {\it fibrewise compactification} of a fibre  space $(Y,p)$ over a locally compact space $X$ is a triple $(Z, \varphi, q)$
where $Z$ is a locally compact space, $q: Z \rightarrow X$ is a continuous
{\it proper} map and $\varphi : Y \rightarrow Z$ is a homeomorphism onto an open
dense subset of $Z$ such that $p = q \circ \varphi$. 
\end{defn}

 We denote by $\cC_0(Y,p)$  the $C^*$-algebra of continuous bounded functions $g$ on $Y$ such that for every $\varepsilon >0$ there exists a compact subset $K$ of $X$ satisfying $\abs{g(y)} \leq \varepsilon$ if $y \notin p^{-1}(K)$.   We denote by $\beta_pY$ the Gelfand spectrum of $\cC_0(Y,p)$. The inclusion  $f\mapsto f\circ p$  from $\cC_0(X)$ into $\cC_0(Y,p)$ defines a surjection $p_\beta$ from $\beta_pY$ onto $X$. It is easily checked that $(\beta_pY,p_\beta)$ is fibrewise compact. We call it the  
{\em Stone-\v Cech fibrewise compactification of} $(Y,p)$. Note that when $X$ is compact, then $\cC_0(Y,p)$ is the $C^*$-algebra of continuous bounded functions on $Y$ and $\beta_pY$ is the usual Stone-\v Cech  compactification $\beta Y$ of $Y$. 

We observe that even if $p: Y\to X$ is open, its extension $p_\beta : \beta_pY\to X$ is not always open. Consider for instance $Y =( [0,1]\times \set{0}) \sqcup (]1/2, 1]\times \set{1}) \subset \R^2$ and let $p$ be the first projection on $X= [0,1]$. Then $\beta_p Y = \beta Y =( [0,1]\times \set{0}) \sqcup (\beta]1/2, 1]\times \set{1}$). The fibres of $\beta_p Y$ are the same as those of $Y$ except $(\beta_p Y)^{1/2} = (\set{1/2}\times \set{0})\sqcup \big((\beta]1/2,1]\setminus ]1/2,1])\times\set{1}\big)$.  Then $\beta_pY \setminus ([0,1]\times \set{0})$ is open and its image by $p_\beta$ is $[1/2,1]$.

The next proposition shows that $(\beta_p Y, p_\beta)$ is the solution of a universal problem.

\begin{prop}\label{prop:universal}\cite[Proposition A.4]{AD16} Let $(Y,p)$ and $(Y_1,p_1)$ be two fibre spaces over $X$, where $(Y_1,p_1)$ is fibrewise compact. Let $\varphi_1: (Y,p) \to (Y_1,p_1)$ be a morphism. There exists 
a unique continuous map  $\Phi_1 : \beta_p Y \to Y_1$ which extends $\varphi_1$. Moreover, $\Phi_1$ is proper and is a morphism of fibre spaces, that is,  $p_\beta = p_1\circ \Phi_1$.
\end{prop}

We assume now that $(Y,p)$ is a $\cG$-space. A {\it $\cG$-equivariant fibrewise compactification} of the $\cG$-space $(Y,p)$ is a fibrewise compactification $(Z,\varphi, q)$ of $(Y,p)$ such that $(Z,q)$ is a $\cG$-space satisfying $\varphi(\gamma y) = \gamma\varphi(y)$ for every $(\gamma, y)\in \cG\,_s\!*_pY$.

We need to extend the $\cG$-action on $(Y,p)$  to a continuous $\cG$-action on $(\beta_p Y,p_\beta)$. Even in the case of a non-discrete group action $G\actson Y$  this is not possible in general: we have to replace $\beta Y$ by the spectrum of the $C^*$-algebra of bounded left-uniformly continuous functions on $G$ \cite{AD02}. 
In the groupoid case it is more complicated, and {\bf we will limit ourselves to the case of \'etale groupoids}.

\begin{prop}\label{prop:max_min} \cite[Proposition 2.5]{AD16} Let $(Y,p)$ be a $\cG$-space, where $\cG$ is an \'etale groupoid. The structure of $\cG$-space of $(Y,p)$ extends in a unique way to the  Stone-\v Cech fibrewise compactification $(\beta_p Y, p_\beta)$ and makes  it a $\cG$-equivariant fibrewise compactification.
\end{prop}

\begin{prop}\label{prop:universal1}\cite[Proposition 2.6]{AD16}  Let $\cG$ be an \'etale groupoid and $(Y,p)$, $(Y_1,p_1)$ be two $\cG$-spaces. We assume that $(Y_1,p_1)$ is fibrewise compact.  Let $\varphi_1: (Y,p) \to (Y_1,p_1)$ be a $\cG$-equivariant morphism. The  unique continuous map  $\Phi_1 : \beta_p Y \to Y_1$ which extends $\varphi_1$ is $\cG$-equivariant. 
\end{prop}

\subsection {Amenability at infinity for \'etale groupoids} We view the fibrewise space $r:\cG \to \cG^{(0)}$ in an obvious way as a left $\cG$-space. Its $\cG$-equivariant fibrewise compactification $(\beta_r \cG, r_\beta)$ will play an important  role in the sequel because of the following observation.

 \begin{prop}\label{prop:SC} An \'etale groupoid $\cG$ is strongly amenable at infinity if and only if the Stone-\v Cech fibrewise compactification $(\beta_r \cG, r_\beta)$ is an amenable $\cG$-space.
\end{prop}

\begin{proof} In one direction, we note that the inclusions $\cG^{(0)}\subset \cG\subset \beta_r \cG$ provide a continuous section for $r_\beta$ and therefore the amenability of the $\cG$-space $\beta_r\cG$ implies the strong amenability at infinity of $\cG$. Conversely, assume that $(Y,p,\sigma)$ satisfies the conditions of Definition \ref{def:ameninf}. We define a continuous $\cG$-equivariant morphism $\varphi : (\cG,r)\to (Y,p)$ by
$$\varphi(\gamma) = \gamma \sigma\circ s(\gamma).$$
Then, by Proposition \ref{prop:universal1}, $\varphi$ extends in a unique way to a continuous $\cG$-equivariant morphism $\Phi$ from $(\beta_r\cG, r_\beta)$ into $(Y,p)$.
Note that $\Phi(\beta_r\cG)$ is a closed $\cG$-invariant subset of $Y$. Now, it follows from \cite[Proposition 2.2.9 (i)]{AD-R} that $\cG\actson \beta_r\cG$ is amenable, since $\cG\actson \Phi(\beta_r\cG)$ is amenable.
\end{proof}

The space $\beta_r\cG$ has the serious drawback that it is not second countable in most of the cases but however it is $\sigma$-compact when $\cG$ is second countable. On the other hand, it has the advantage of being intrinsic. Moreover it is possible to build a {\it second countable} amenable fibrewise compact $\cG$-space out of any amenable fibrewise compact $\cG$-space, when $\cG$ is second countable  and \'etale \cite[Lemma 4.9]{AD16}.

\begin{thm}\label{prop:amen_inf} Let $\cG$ be a second countable \'etale groupoid. The following conditions are equi\-valent:
\begin{itemize}
\item[(i)] $\cG$ is strongly amenable at infinity;
\item[(ii)]  there exists a sequence $(k_n)$ of  bounded positive definite  continuous
kernels on $\cG *_r \cG$ supported in tubes such that
\begin{itemize}
\item[(a)] for every $n$, the restriction of $k_n$ to the diagonal of $\cG *_r \cG$ is uniformly bounded by $1$;
\item[(b)]  $\lim_n k_n = 1$ uniformly on tubes.
\end{itemize}
\end{itemize} 
\end{thm}

\begin{proof} By Theorem \ref{thm:caractAmen}, the groupoid $\beta_r\cG\rtimes \cG$ is amenable if and only if there exists a net $(h_n)$ of continuous positive definite functions in $\cC_c(\beta_r\cG\rtimes \cG)$, whose restriction to the set of units are bounded by 1, such that 
 $\lim_i h_i = 1$ uniformly on every compact subset of $\beta_r\cG\rtimes \cG$. 
 
 For $(\gamma_1,\gamma_2)\in \cG*_r\cG$ we set $k_n(\gamma_1,\gamma_2) = h_n(\gamma_1^{-1}, \gamma_1^{-1}\gamma_2)$. Then we check that $k_n$ is a  positive definite kernel bounded by $1$ on the diagonal, supported in a tube, and that $\lim_i k_n = 1$ uniformly on tubes.
 
 The converse is proved similarly (see \cite[Theorem 4.13, Theorem 4.15]{AD16} for details).  
\end{proof}

As observed in \cite[Remark 3.4, Remark 4.16]{AD16} it suffices in (ii) (a) above to require that each $k_n$ is bounded.

\section{About exactness for groupoids}

\subsection{Equivalence of several definitions of exactness for  \'etale groupoids} 

\begin{defn}
Let $\cG$ be a locally compact groupoid. We say that $\cG$ is  KW{\it-exact} if for  every $\cG$-equivariant exact sequence
$0\to I \to A \to B \to 0$
of $\cG$-$C^*$-algebras, the corresponding sequence
$$0 \to C^*_{r}(\cG,I) \to C^*_{r}(\cG,A) \to C^*_{r}(\cG,B)\to 0$$
of reduced crossed products is exact.
We say that $\cG$  is {\it $C^*$-exact} if $C^*_{r}(\cG)$ is an exact $C^*$-algebra. 
\end{defn}

For the definition of actions of locally compact groupoids on $C^*$-algebras and the construction of the corresponding crossed products we refer for instance to \cite{KS04} or \cite[\S 6.2]{AD16}.  As in the case of groups we easily see that amenability at infinity  implies KW-exactness which in turn  implies $C^*$-exactness (see for instance \cite[Theorem 7.2]{AD02} for groups and \cite[\S 7]{AD16} for groupoids). The main problem is to see if $C^*$-exactness of an \'etale groupoid implies its amenability at infinity as it is the case for discrete groups.

In this section we will adapt to the \'etale groupoid case our proof of the fact that an inner amenable locally compact $C^*$-exact group is amenable at infinity \cite[Theorem 7.3]{AD02}.  We need first to define inner amenability for groupoids.

\begin{defn} Let $\cG$ be a locally compact groupoid. Following \cite[Definition 2.1]{Roe}, we say that a closed subset $A$ of $\cG \times \cG$ is {\it proper}  if for every compact subset $K$ of $\cG$, the sets $(K\times \cG) \cap A$ and $(\cG \times K) \cap A$ are compact. We say that a function $f: \cG\times \cG \to \C$ is {\it properly supported} if its support is  proper.
\end{defn}

Given a groupoid $\cG$, let us observe that the product $\cG\times \cG$ has an obvious structure of groupoid, with $X\times X$ as set of units, where $X= \cG^{(0)}$. Observe that a map $f:\cG\times \cG\to \C$ 
is positive definite if and only if, given an integer $n$, $(x,y)\in X\times X$ and $\gamma_1,\dots, \gamma_n\in \cG^x$, $\eta_1,\dots,\eta_n\in \cG^y$, the matrix $[f(\gamma_i^{-1}\gamma_j, \eta_i^{-1}\eta_j)]_{i,j}$ is non-negative.

\begin{defn}\label{def:wia} We say that a locally compact groupoid $\cG$ is {\it  inner amenable}\index{inner amenable l. c. groupoid}  if for every compact subset $K$ of $\cG$ and for every $\varepsilon >0$
there exists a continuous  positive definite function $f$ on the product groupoid $\cG\times \cG$, properly supported, such that $f(x,y)\leq 1$ for all $x,y\in \cG^{(0)}$ and such that $|f(\gamma,\gamma) - 1| < \varepsilon$ for all $\gamma \in K$.
\end{defn}

This terminology is justified by the fact that for a  locally compact group the above property is equivalent to the notion of inner amenability introduced in Section 1. That this property for groups implies inner amenability is proved in \cite{CT}; the reverse is almost immediate \cite[Proposition 4.6]{AD02}. 

Every amenable locally compact groupoid $\cG$ is  inner amenable since the groupoid $\cG\times \cG$ is amenable and therefore Theorem \ref{thm:caractAmen} applies to this groupoid. Every closed subgroupoid of an inner amenable groupoid is inner amenable \cite[Corollary 5.6]{AD16}.
Every semidirect product groupoid $X\rtimes G$  is inner amenable as soon as $G$ is an inner amenable locally compact group  \cite[Corollary 5.9]{AD16}. We do not know whether every \'etale groupoid is inner amenable.

\begin{thm}\label{thm:equiv} Let $\cG$ be a second countable  inner amenable \'etale groupoid. Then the following conditions are equivalent:
\begin{itemize}
\item[(1)] $\cG$ is strongly amenable at infinity.
\item[(2)] $\cG$ is  amenable at infinity.
\item[(3)] $\beta_r \cG\rtimes \cG$ is nuclear.
\item[(4)] $\beta_r \cG\rtimes \cG$ is exact.
\item[(5)] $\cG$ is KW-exact.
\item[(6)] $C^*_{r}(\cG)$ is exact.

\end{itemize}
\end{thm}

The following implications are immediate or already known:
$$\xymatrix{(1) \ar@{=>}[r] \ar@{=>}[d] &  (2)\ar@{=>}[r] &(5)\ar@{=>}[d]\\
(3) \ar@{=>}[r] & (4) \ar@{=>}[r] & (6)} $$
The implication (1) $\Ra$ (3) is proved in \cite[Corollary 6.2.14]{AD-R} for second countable locally compact groupoids, but this result extends to the groupoid $\beta_r \cG\rtimes \cG$ when $\cG$ is second countable locally compact and \'etale (see \cite[Proposition 7.2]{AD16}).

It remains  to show that (6) implies (1). We give below an idea of the proof which is detailed in \cite[\S 8]{AD16}. 

$\blacktriangleright$ The first step is to extend Kirchberg's characterization of exact $C^*$-algebras as being nuclearly embeddable into some 
$\cB(H)$ as follows.

\begin{lem}\cite[Lemma 8.1]{AD16}\label{lem:Kirch} Let $A$, $B$ be two separable $C^*$-algebras, where $B$ is nuclear. Let $\cE$ be a countably generated Hilbert $C^*$-module over $B$.  Let $\iota : A \to \cB_B(\cE)$ be an embedding of $C^*$-algebras. Then $A$ is exact if and only if $\iota$ is nuclear.
\end{lem}

The two main ingredients of the proof of this lemma are the Kasparov absorption theorem and the Kasparov-Voiculescu theorem \cite[Theorem 2, Theorem 6]{Kasp} that allow us to  reduce the situation to the case of Hilbert spaces.

$\blacktriangleright$ The second step is the following approximation theorem. Recall that  a completely positive contraction $\Phi:A\to B$  between two $C^*$-algebras is {\it factorable} if there exists an integer $n$ and completely positive contractions $\psi: A \to M_{n}(\C)$, $\varphi :  M_{n}(\C) \to B$ such that $\Phi = \varphi \circ \psi$. A map $\Psi :C^*_{r}(\cG) \to B$ is said to have a {\it compact support} if there exists a compact subset $K$ of $\cG$ such that $\Psi(f) = 0$ for every $f \in \cC_c(\cG)$ with $(\hbox{Supp}\, f) \cap K = \emptyset$.

\begin{thm}\cite[Corollary 8.4]{AD16}\label{cor:approx} Let $B$ be a $C^*$-algebra and let $\Phi :  C^{*}_{r}(\cG) \to B$ be a nuclear completely positive contraction.
Then for every $\varepsilon > 0$ and every $a_1,\dots,a_k \in  C^{*}_{r}(\cG)$
there exists a factorable   completely positive contraction $\Psi :  C^{*}_{r}(\cG) \to B$, with  compact support,
such that  $$\| \Psi(a_i) - \Phi(a_i) \| \leq \varepsilon \,\,\, \hbox{for}\,\,\, i = 1,\dots,k.$$
\end{thm}

$\blacktriangleright$ Finally we need the following result due to Jean Renault (private communication). Given $f:\cG\times \cG \to \C$, we set $f_\gamma(\gamma') = f(\gamma, \gamma')$.

\begin{lem}\cite[Lemma 8.5]{AD16} Let $\cG$ be a locally compact groupoid.
\begin{itemize}
\item[(a)] Let $f\in \cC_c(\cG)$ be a continuous positive definite function. Then, $f$ viewed as an element of $C^*_{r}(\cG)$ is a positive element.
\item[(b)] Let $f:\cG\times \cG \to \C$ be a properly supported positive definite function. Then $\gamma\mapsto f_\gamma$ is a continuous  positive definite function  from $\cG$ into  $C^*_{r}(\cG)$.
\end{itemize}
\end{lem}

$\blacktriangleright$ We can now now proceed to the proof of (6) $\Ra$ (1)  in Theorem \ref{thm:equiv}.

\begin{proof}[Proof of (6) $\Ra$ (1)]

We fix a compact subset $K$ of $\cG$ and $\varepsilon >0$. We want to find a continuous bounded positive definite kernel $k\in \cC_t(\cG*_r\cG)$ such that $k(\gamma,\gamma) \leq 1$ for all $\gamma\in \cG$ and $\abs{k(\gamma,\gamma_1) - 1} \leq \varepsilon$ whenever $\gamma^{-1}\gamma_1 \in K$ (see Theorem \ref{prop:amen_inf}).

We set $\cE =  \ell^2_{\cC_0(X)}(\cG)$ with $X = \cG^{(0)}$. Recall that $\lambda^x$ is the counting measure on $\cG^x$. We first choose a bounded, continuous positive definite function $f$ on $\cG\times \cG$, properly supported,
such that $|f(\gamma,\gamma) - 1| \leq \varepsilon/2$ for $\gamma \in K$ and $f(x,y)\leq 1$ for $(x,y)\in X\times X$. By Lemma \ref{lem:Kirch} the regular representation $\Lambda$ is nuclear. Then, using Theorem \ref{cor:approx}, we find  a compactly supported completely positive contraction  $\Phi: C^*_{r}(\cG) \to \cB_{\cC_0(X)}(\cE)$   such that\footnote{We write $f_\gamma$ instead of $\Lambda(f_\gamma)$
for simplicity of notation.}
$$\norm{\Phi(f_\gamma) - f_\gamma} \leq \varepsilon/2$$
for $\gamma \in K$. 
We also choose a continuous function $\xi : X \to [0,1]$ with compact support such that $\xi(x) = 1$ if $x\in s(K)\cup r(K)$. 

Let $(\gamma,\gamma_1)\in \cG*_r\cG$. We choose an open bisection $S$  such that $\gamma\in S$ and a continuous function $\varphi: X\to [0,1]$, with compact support in $r(S)$ such that $\varphi (x) = 1$ on a neighborhood of $r(\gamma)$. We denote by $\xi_\varphi$ the continuous function on $\cG$ with compact support  (and thus $\xi_\varphi \in \cE$) such that
$$\xi_\varphi(\gamma') = 0 \,\,\hbox{if}\,\, \gamma'\notin S,\quad \xi_\varphi(\gamma') = \varphi\circ r(\gamma')\xi\circ s(\gamma') \,\,\hbox{if}\,\,\gamma'\in S.$$
Note that $\norm{\xi_\varphi}_\cE \leq 1$.
We define $\xi_{\varphi_1}$ similarly with respect to $\gamma_1$. 

Then we set
\begin{align*}
k(\gamma,\gamma_1) &=  \langle \xi_\varphi,
\Phi(f_{\gamma^{-1}\gamma_{1}})\xi_{\varphi_1}\rangle (r(\gamma))\\
&= \xi\circ s(\gamma)\big(\Phi(f_{\gamma^{-1}\gamma_{1}})\xi_{\varphi_1}\big)(\gamma).
\end{align*}
We observe that $k(\gamma,\gamma_1)$ does not depend on the choices of $S,\varphi, S_1,\varphi_1$.

Since $\gamma \mapsto f_\gamma$ is a continuous positive definite function from $\cG$ into $C^*_{r}(\cG)$ and since $\Phi$ is completely positive,  we see that $k$ is  a continuous and positive definite kernel. Moreover, there is a compact subset $K_1$ of $\cG$ such that $\Phi(f_\gamma) = 0$ when $\gamma\notin K_1$, because  $\Phi$ is compactly supported, and  $f$  is properly supported.  It follows that $k$ is  supported in a tube.
 
We fix  $(\gamma,\gamma_1)\in \cG*_r\cG$  such that $\gamma^{-1}\gamma_1 \in K$. Then we have
$$\abs{k(\gamma,\gamma_1) -1}\leq \varepsilon/2 + \abs{ \langle \xi_\varphi,f_{\gamma^{-1}\gamma_{1}}\xi_{\varphi_1}
\rangle(r(\gamma)) -1},$$
and
$$
 \langle \xi_\varphi, f_{\gamma^{-1}\gamma_{1}}\xi_{\varphi_1}\big\rangle(r(\gamma))\\
=  \xi\circ s(\gamma)\xi\circ s(\gamma_1)f(\gamma^{-1}\gamma_{1},\gamma^{-1}\gamma_{1}).
$$

Observe that $s(\gamma)\in r(K)$ and $s(\gamma_1)\in s(K)$ and therefore $ \xi\circ s(\gamma) = 1=  \xi\circ s(\gamma_1)$.
It follows that 
$$\abs{k(\gamma,\gamma_1) -1}\leq \varepsilon/2 + \abs{f(\gamma^{-1}\gamma_{1},\gamma^{-1}\gamma_{1})-1}\leq \varepsilon.$$
To end the proof it remains to check that $k$ is a bounded kernel. Since this kernel is positive definite, it suffices to show that $\gamma\mapsto k(\gamma,\gamma)$ is bounded on $\cG$. We have
 $$k(\gamma,\gamma) =  \langle \xi_\varphi,\Phi(f_{s(\gamma)})\xi_{\varphi}\rangle (r(\gamma))  \leq \norm{\Phi(f_{s(\gamma)})}.$$
  Our claim follows, since  $\Phi(f_{s(\gamma)})= 0$ when $s(\gamma)\notin K_1\cap X$ and $x\mapsto f_x$ is continuous from the compact set  $K_1\cap X$ into $C_r^{*}(\cG)$. 
\end{proof}

\begin{rem}\label{rem:partial} Let $\alpha : \Gamma\actson X$ be an action of a discrete group  on a locally compact space $X$.  Since the groupoid $X\rtimes G$ is inner amenable, Theorem \ref{thm:equiv} applies. Therefore $C_0(X)\rtimes_r \Gamma$ is exact if and only if the groupoid $X\rtimes \Gamma$ is KW-exact.
More generally, this holds for any partial action such that the domains of the partial homeomorphisms $\alpha_t$ are closed (in addition to being open). Indeed it not difficult to show that  the groupoid $X\rtimes \Gamma$ is inner amenable (directly, or using the fact that such partial actions admit a Hausdorff globalisation   \cite[Proposition 5.7]{Exel_Book}). 

For general partial actions of $\Gamma$ the situation is not clear. We do not know whether $X\rtimes \Gamma$ is inner amenable in this case. If 
$\Gamma$ is exact the semidirect product groupoid $X\rtimes \Gamma$ is strongly amenable at infinity \cite[Proposition 4.23]{AD16} and therefore $C^*_{r}(X\rtimes \Gamma)$ is exact. This had been previously shown in \cite[Corollary 2.2]{AEK} by using  Fell bundles. 
\end{rem}

\subsection{Inner exactness} We introduce now a very weak notion of exactness.  First let us make some reminders. Let $\cG$ be a locally  compact groupoid. Recall that a subset $E$ of $X= \cG^{(0)}$ is said to be {\it invariant} if $s(\gamma)\in E$ if and only if $r(\gamma)\in E$. 
Let $F$ be a closed invariant subset of $X$ and set $U = X\setminus F$. It is well-known that the inclusion
$\iota :  \cC_c(\cG(U) )\to \cC_c(\cG)$ extends to an injective homomorphism from $C^*(\cG(U))$ into $C^*(\cG)$ and from
$C^*_{r}(\cG(U))$ into $C^*_{r}(\cG)$. Similarly, the restriction map $\pi: \cC_c(\cG) \to \cC_c(\cG(F))$ extends to a surjective homomorphism from $C^*(\cG)$ onto $C^*(\cG(F))$ and from
$C^*_{r}(\cG)$ onto $C^*_{r}(\cG(F))$. Moreover the sequence 
$$0 \rightarrow C^*(\cG(U)) \rightarrow C^*(\cG) \rightarrow C^*(\cG(F)) \rightarrow 0$$
is exact. For these facts, we refer to \cite[page 102]{Ren_book},  or to \cite[Proposition 2.4.2]{Ram} for a detailed proof. On the other hand, the corresponding sequence 
\begin{equation}\label{eq:ie}
0 \rightarrow C^*_{r}(\cG(U)) \rightarrow C^*_{r}(\cG) \rightarrow C^*_{r}(\cG(F)) \rightarrow 0
\end{equation}
with respect to the reduced groupoid $C^*$-algebras is not always exact, as shown in \cite[Remark 4.10]{Ren91} (see also Proposition \ref{prop:HLS} below).

\begin{defn}\label{def:inamen} A locally compact groupoid  such that the sequence  \eqref{eq:ie} is exact for every closed invariant subset $F$ of $X$ called KW-{\it inner exact} or simply {\it inner exact}.\index{inner exact groupoid}
\end{defn}
 
We will see that the class of inner exact groupoids plays a role in the study of the (WCP). It is  also interesting in itself and now plays a role in other contexts (see for instance \cite{BL}, \cite{BCS}, \cite{BEM}).  This class is quite large. It includes all locally compact groups and more generally the groupoids that act with dense orbits on their space of units. This class is stable under equivalence of groupoids \cite[Theorem 6.1]{Lal17}. Of course, KW-exact groupoids are inner exact.

\subsection{The case of group bundle groupoids} We first need to recall  some definitions.

\begin{defn}\label{def:bundle_C*}\cite[Definition 1.1]{KW95}  A {\it field} (or bundle)
{\it of $C^*$-algebras over a locally compact space} $X$ is a triple $\cA = (A, \set{\pi_x: A \to A_x}_{x\in X},X)$ where $A$, $A_x$ are $C^*$-algebras, and where $\pi_x$ is a surjective $*$-homomorphism such that 
\begin{itemize}
\item[(i)] $\set{\pi_x: x\in X}$ is faithful, that is, $\norm{a} = \sup_{x\in X}\norm{\pi_x(a)}$ for every $a\in A$;
\item[(ii)] for $f\in \cC_0(X)$ and $a\in A$, there is an element $fa\in A$ such that $\pi_x(fa) = f(x)\pi_x(a)$ for $x\in X$;
\item[(iii)] the inclusion of $\cC_0(X)$ into the center of the multiplier algebra of $A$ is non-degenerate.

\end{itemize}
We say that the field is (usc) {\it upper semi-continuous} (resp. (lsc) {\it lower semi-continuous}) if the function $x\mapsto \norm{\pi_x(a)}$ is upper semi-continuous (resp. lower semi-continuous) for every $a\in A$.

If for each $a\in A$, the function $x\mapsto \norm{\pi_x(a)}$ is in $\cC_0(X)$, we will say that $\cA$ is a {\it continuous field of $C^*$-algebras}\footnote{In  \cite{KW95}, this is called a continuous bundle of $C^*$-algebras.}.
\end{defn}

Recall that a $\cC_0(X)$-algebra  $A$ is a $C^*$-algebra equipped with a non-degenerate homomorphism  from $\cC_0(X)$ into the multiplier algebra of $A$ (see \cite[Appendix C.1]{Will}).
 For $x\in X$ we denote by $\cC_x(X)$ the subalgebra of $\cC_0(X)$ of functions that vanish at $x$. Note that a $\cC_0(X)$-algebra  $A$ gives rise to an usc field of $C^*$-algebras with fibres $A_x = A/\cC_x(X)A$ (see \cite[Proposition 1.2]{Rie} or \cite[Appendix C.2]{Will}). We will use the following characterization of usc fields of $C^*$-algebras.
 
\begin{lem}\cite[Lemma 2.3]{KW95}, \cite[Lemma 9.4]{AD16}\label{lem:usc} Let $\cA$ be a field of $C^*$-algebras on a locally compact space $X$.  The function $x\mapsto \norm{\pi_x(a)}$ is upper semi-continuous at $x_0$ for every $a\in A$ if and only if $\ker \pi_{x_0}= \cC_{x_0}(X)A$
\end{lem}

 We apply this fact to the reduced $C^*$-algebra of a  groupoid group bundle $\cG$  as defined in Example \ref{exs:groupoids}  (b). The structure of $\cC_0(X)$-algebra of  the $C^*$-algebra $C^*_{r}(\cG)$ is defined by $(f h)(\gamma) = f\circ r(\gamma) h(\gamma)$ for $f\in \cC_0(X)$ and $h\in \cC_c(\cG)$ (see \cite[Lemma 2.2.4]{Ram}, \cite[\S 5]{LR}).   We set $U_x = X \setminus \set{x}$. Then we have $C^*_{r}(\cG(U_x)) = \cC_x(X) C^*_{r}(\cG)$.   We get that $C^*_{r}(\cG)$   is an usc  field of $C^*$-algebras over $X$ with fibre $C^*_{r}(\cG)/\cC_x(X) C^*_{r}(\cG)= C^*_{r}(\cG)/C^*_{r}(\cG(U_x)) $ at $x$.

On the other hand, $(C^*_{r}(\cG), \set{\pi_x : C^*_{r}(\cG)\to C^*_{r}(\cG(x))})$ is lower semi-continuous (see 
 \cite[Th\'eor\`eme 2.4.6]{Ram} or \cite[Theorem 5.5]{LR}). Then it follows from Lemma \ref{lem:usc} that  the function $x\mapsto \norm{\pi_x(a)}$ is continuous at $x_0$ for every $a\in C^*_{r}(\cG)$ if and only if  the following sequence  is exact:
  $$0\rightarrow C^*_{r}(\cG(U_{x_0})) \rightarrow C^*_{r}(\cG) \stackrel{\pi_{x_0}}{\rightarrow} C^*_{r}(\cG(x_0)) \rightarrow 0.$$

\begin{prop}\label{prop:innexact} Let $\cG$ be a  group bundle groupoid  on $X$. The following conditions are equi\-valent:
\begin{itemize}
\item[(i)] $\cG$ is inner exact;
\item[(ii)] for every $x\in X$ the following sequence is exact:
$$
0 \rightarrow C^*_{r}(\cG(X\setminus \set{x})) \rightarrow C^*_{r}(\cG) \rightarrow C^*_{r}(\cG(x)) \rightarrow 0.
$$

\item[(iii)] $C^*_{r}(\cG)$ is a continuous field of $C^*$-algebras over $X$ with fibres $C^*_{r}(\cG(x)) $.
\end{itemize}
\end{prop}

\begin{proof} (i) $\Rightarrow$ (ii) is obvious and   (ii) $\Rightarrow$ (iii) is a particular case of the previous observation. Assume that (iii) holds true and, given an invariant closed subset $F$ of $X$,  let us show that the following sequence  is exact:
$$
0 \rightarrow C^*_{r}(\cG(X\setminus F)) \rightarrow C^*_{r}(\cG) \rightarrow C^*_{r}(\cG(F)) \rightarrow 0.
$$

Let $a\in C^*_{r}(\cG)$ be such that $\pi_x(a) = 0$ for every $x\in F$. Let $\varepsilon >0$ be given. Then $K = \set{x\in X, \norm{\pi_x(a)} \geq \varepsilon}$ is a compact subset of $X$ with $K\cap F = \emptyset$. Take $f\in \cC_0(X), f: X\to [0,1]$ with $f(x) = 1$ for $x\in K$ and $f(x) = 0$ for $x\in F$. We have $\norm{a-fa}\leq \varepsilon$  and $fa\in C^*_{r}(\cG(X\setminus F))$. Therefore $a\in C^*_{r}(\cG(X\setminus F))$.
\end{proof}

\subsection{The case of HLS groupoids} \label{sec:HLS}
The following class of \'etale group bundle groupoids  (that we call HLS-groupoids)\index{HLS-groupoid} was introduced by  Higson, Lafforgue and Skandalis \cite{HLS},
 in order to provide examples of groupoids for which the Baum-Connes conjecture fails. We consider an infinite discrete group $\Gamma$ and a decreasing sequence $(N_k)_{k\in \N}$ of normal subgroups of $\Gamma$ of finite index. We set $\Gamma_\infty = \Gamma$, and $\Gamma_k = \Gamma/N_k$ and we denote by $q_k : \Gamma \to \Gamma_k$ the quotient homomorphism for $k$ in the  Alexandroff compactification $\N^+$ of $\N$. Let  $\cG$ be the quotient of $  \N^+ \times\Gamma$ with respect to the equivalence relation
$$(k,t)\sim (l,u) \,\,\,\hbox{if} \,\,\, k=l \,\,\,\hbox{and}\,\,\, q_k(t) = q_k(u).$$ 
Then $\cG$ is the bundle of groups $k\mapsto \Gamma_k$ over $\N^+$. The range and source maps are given by $r([k,t]) = s([k,t]) = k$, where $[k,t] = (k,q_k(t))$ is the equivalence class of $(k,t)$.  We endow $\cG$ with the quotient topology.  Then $\cG$ is Hausdorff (and obviously an \'etale groupoid) if and only if  for every $s\not= 1$ there exists $k_0$ such that $s\notin N_k$ for $k\geq k_0$ (hence, $\Gamma$ is residually finite).  We keep this assumption. Such examples are provided by taking $\Gamma = \hbox{SL}_n(\Z)$ and $\Gamma_k = \hbox{SL}_n(\Z/k\Z)$, for $k\geq 2$.

For these HLS groupoids, the exactness of $C^*_{r}(\cG)$ is a very strong condition which suffices to imply the amenability of $\Gamma$ as shown by Willett in \cite{Wil15}.

\begin{prop}\label{prop:HLS} Let us keep the above notation. We assume that $\Gamma$ is finitely generated. Then the following conditions are equivalent:
\begin{itemize}
\item[(1)] $\Gamma$ is amenable; $(2)$ $\cG$ is amenable;
\item[(3)]  $\cG$ is KW-exact; $(4)$ $\cG$ is  inner exact;
 \item[(5)] the sequence 
 $0 \longrightarrow C^*_{r}(\cG(\N))\longrightarrow C^*_{r}(\cG) \longrightarrow C^*_{r}(\cG(\infty)) \longrightarrow 0$
 is exact $($$($5'$)$ $C^*_{r}(\cG)$ is a continuous field of  $C^*$-algebras with fibres $C^*_{r}(\cG(x))$$)$;
\item[(6)] $C^*_{r}(\cG)$ is nuclear; $(7)$ $C^*_{r}(\cG)$ is exact.
\end{itemize}
\end{prop}

\begin{proof} The equivalence between (1) and  (2)  follows for instance from \cite[Lemma 2.4]{Wil15}. That (2) $\Ra$ (3) $\Ra$ (4) $\Ra$ (5) is obvious and by Proposition \ref{prop:innexact} we have (5) $\Ra$ (5'). Let us prove that (5') $\Ra$ (1). Assume by contradiction that $\Gamma$ is not amenable. We fix a symmetric probability measure $\mu$ on $\Gamma$ with a finite support that generates $\Gamma$ and we choose $n_0$ such that the restriction of $q_n$ to the support of $\mu$ is injective for $n\geq n_0$. We take $a\in\cC_c(\cG)\subset C^*_{r}(\cG)$ such that $a(\gamma) = 0$ except for $\gamma = (n,q_n(s))$ with $n\geq n_0$ and $s\in \supp(\mu)$ where $a(\gamma) = \mu(s)$. Then $\pi_n(a) = 0$ if $n<n_0$ and $\pi_n(a) = \lambda_{\Gamma_n}(\mu) \in C_{r}^*(\Gamma_n)= C^*_{r}(\cG(n))$ if $n\geq n_0$, where $\lambda_{\Gamma_n}$ is the quasi-regular representation of $\Gamma$ in $\ell^2(\Gamma_n)$. By Kesten's result \cite{Kes1, Kes} on  spectral radii relative to symmetric random walks, we have $\norm{\lambda_{\Gamma_n}(\mu)}_{C_{r}^*(\Gamma_n)} =1$ for $\N \ni n\geq n_0$ and $\norm{\lambda_{\Gamma_\infty}(\mu)}_{C_{r}^*(\Gamma_\infty)} < 1$ since $\Gamma$ is not amenable. It follows that $C^*_{r}(\cG)$ is not a continuous field of $C^*$-algebras with fibres $C^*_{r}(\cG(n))$ on $\N^+$, a contradiction.

We know that (2) $\Ra$  (6) $\Ra$ (7). For the fact that (7) $\Ra$ (1) see \cite[Lemma 3.2]{Wil15}.
\end{proof}

Given a group bundle groupoid it may happen that $\cG(x)$ is a $C^*$-exact group for every $x\in \cG^{(0)}$ whereas $\cG$ is not $C^*$-exact. Indeed if $\cG$ is an HLS groupoid associated with a group $\Gamma$ that has Kazdhan's property (T), then the sequence
$$0 \longrightarrow C^*_{r}(\cG(\N))\longrightarrow C^*_{r}(\cG) \longrightarrow C^*_{r}(\cG(\infty)) \longrightarrow 0$$
is not exact (it is not even exact in $K$-theory!)  \cite{HLS}.  As an example we can take the exact group $\Gamma = SL(3,\Z)$. The previous proposition shows that $C^*_{r}(\cG)$ is not exact. Willett has given an even more surprising example with $\Gamma = \F_2$, the free group with two generators (see below).

\section{Weak containment, exactness and amenability}
\begin{defn}\label{def:weakamen} Let $\cG$ be a locally compact groupoid. We say that $\cG$ has the {\it weak containment property}, \index{weak containment property}(WCP) \index{(WCP)} in short, if the canonical surjection from its full $C^*$-algebra $C^*(\cG)$ onto its reduced $C^*$-algebra $C^*_{r}(\cG)$ is injective, {\it i.e.}, the two completions $C^*(\cG)$ and $C^*_{r}(\cG)$  of $\cC_c(\cG)$ are the same.
\end{defn}

A very useful theorem of Hulanicki \cite{Hul, Hul66} asserts that  a locally compact group $G$ has the (WCP) if and only if it is amenable. While it has long been known that every amenable locally compact groupoid has the (WCP) \cite{Ren91, AD-R}, whether the converse holds was a long-standing open problem (see \cite[Remark 6.1.9]{AD-R}).   Remarkably, in 2015 Willett \cite{Wil15} published a very nice example showing that a group bundle groupoid may have the (WCP) without being  amenable. His example is a HLS groupoid where $(\Gamma_n)$ is a well chosen sequence of subgroups of the free group with two generators $\F_2$. Therefore the groupoid version of Hulanicki's theorem is not true in general. However there are many positive results,  all of which involve an additional exactness assumption.

 A first result in this direction is due to Buneci \cite{Bun}. She proved that a second countable locally compact transitive groupoid  $\cG$ having the (WCP)  is measurewise amenable.  The (topological) amenability of $\cG$ can also be proved by observing that it is preserved under equivalence of groupoids \cite[Theorem 2.2.17]{AD-R}, as well as the (WCP) \cite[Theorem 17]{SW12}, and using the fact that $\cG$   is equivalent to any of its isotropy group by transitivity \cite{MRW}.
 
 It is only in 2014 that a second result appeared, linking amenability and the (WCP)
 \begin{thm}\label{Mat}\cite{Mat} Let $\Gamma$ be a discrete group acting by homeomorphisms on a compact space. Then the semidirect product groupoid is amenable if and only if it has the (WCP) and $\Gamma$ is exact.
 \end{thm}
 
 Note that the exactness of $\Gamma$ is equivalent to the strong amenability at infinity  of the groupoid $X\rtimes \Gamma$ since $X$ is compact  and $\Gamma$ is discrete (see \cite[Proposition 4.3 (i)]{AD16}). Recently, the above theorem has been extended by Kranz as follows.
 
\begin{thm}\label{Kranz} \cite{Kra} Let $\cG$ be an \'etale  groupoid. Then $\cG$ is amenable if and only if it has the (WCP) and is strongly amenable at infinity.
\end{thm}

Assuming that $\cG$ has the (WCP) and is strongly amenable at infinity, 
Kranz's strategy to prove that $\cG$ is amenable is the same as that of  \cite{Mat}, but with additional technical difficulties. It consists in showing that the canonical inclusion of $C^*_{r}(\cG)$ into its bidual $C^*_{r}(\cG)^{**}$ is nuclear. Then by \cite[Proposition 2.3.8]{BO} one sees that $C^*_{r}(\cG)$ is  nuclear  and by \cite[Theorem 5.6.18]{BO} it follows that the groupoid  $\cG$ is amenable. The delicate step, which requires in a crucial way that $\cG$ is \'etale,  is to show the existence of a completely positive map $\phi: C^*_{r}(\beta_r\cG \rtimes \cG) \to C^*_{r}(\cG)^{**}$ whose restriction to $C^*_{r}(\cG)$ is the inclusion from $C^*_{r}(\cG)$ in its bidual. Since $C^*_{r}(\beta_r\cG \rtimes \cG)$ is nuclear (see \cite[Proposition 7.2]{AD16}), this inclusion is nuclear.

 By a different method  the following extension of Theorem \ref{Mat} was obtained in \cite{BEW20}. Note that unlike the case where $X$ is compact it is not true in general that  $G$ is KW-exact when $X\rtimes G$ is amenable.
 \vspace{1mm}

 \begin{thm} \cite[Theorem 5.15]{BEW20} Let $G\actson X$ be a continuous action of a locally compact group $G$ on a locally compact space $X$. We assume that $G$ is KW-exact and that $X\rtimes G$ has the (WCP). Then the groupoid $X\rtimes G$ is amenable.
 \end{thm}

 It is interesting to note that the behaviour is different for group actions on non-commutative $C^*$-algebras. For instance in \cite[Proposition 5.25]{BEW20} a surprising  example of a non-amenable action  having the (WCP) of the exact group $G= PSL(2,\C)$ on the $C^*$-algebra of compact operators  has been constructed. It would be interesting to construct an example with an exact discrete group.

 For group bundle groupoids  we have the following easy result.
 
 \begin{prop}\label{prop:WCPbundle} Let $\cG$ be a second countable locally compact  group bundle groupoid over a locally compact space $X$. Then $\cG$   is amenable  if and only if it has the (WCP) and is inner exact.
\end{prop}

\begin{proof} Assume that $\cG$ has the (WCP) and is inner exact and let $x\in X$. We set $U_x = X\setminus \set{x}$. In the commutative diagram
$$\xymatrix{
0\ar[r] &C^*(\cG(U_x)) \ar[d] \ar[r] & C^*(\cG)\ar[d]^{\lambda} \ar[r] &C^* (\cG(x))
\ar[r]\ar[d]^{\lambda_{\cG(x)}} &0\\
0\ar[r] & C^*_{r}(\cG(U_x)) \ar[r] & C^*_{r}(\cG)\ar[r]^{\pi_x}&  C^*_r (\cG(x))
\ar[r] &0}$$
both sequences are exact and $\lambda$ is injective. Chasing through the diagram we see  that $\lambda_{\cG(x)}$ is injective ({\it i.e.}, the group $\cG(x)$ is amenable).  This ends the proof since, by \cite[Theorem 3.5]{Ren13}, the  group bundle groupoid $\cG$ is amenable if and only if $\cG(x)$ is amenable for every $x\in X$.
\end{proof}

The cases of transitive groupoids and of group bundle groupoids are included in the following result of B\"onicke. His   nice elementary  proof is reproduced in \cite[Theorem 10.5]{AD16}.

\begin{thm}\cite{Bon}\label{thm:bon} Let $\cG$ be a  second countable locally compact groupoid  such that the orbit space $\cG\setminus \cG^{(0)}$ equipped with the quotient topology is $T_0$. Then the following conditions are equivalent:
\begin{itemize}
\item[(i)] $\cG$ is  amenable;
\item[(ii)] $\cG$ has the (WCP) and is inner exact.
\end{itemize}
\end{thm}

\section{Open questions}
\subsection{About  amenability at infinity and inner amenability}
\begin{itemize}
\item[(1)] The notion of strong amenability at infinity has proven to be more useful than amenability at infinity. But are the two notions equivalent?  Note that by Theorem \ref{thm:equiv}
this is true for every second countable  inner  amenable \'etale groupoid.
\vspace{2mm}

\item[(2)] It would be interesting to understand better the notion of inner amenability for locally compact groupoids.  Is it invariant under equivalence of groupoids? \footnote{It is at least true for \'etale groupoids.} Are there \'etale groupoids that are not  inner amenable? In particular, if $G$ is a discrete group acting partially on a locally compact space X, is it true that the corresponding partial transformation groupoid is  inner amenable? This is true when the domains of the partial homeomorphisms are both open and closed but what happens in general? It would also be interesting to study the case of HLS groupoids.
\end{itemize}

\subsection{About exactness for groups} 
\begin{itemize}

\item[(3)] Let us denote by $[InnAmen]$ the class of locally compact inner amenable groups and by $[Tr]$ the class of groups whose reduced $C^*$-algebra has a tracial state. Groups $G$ in each of these two classes and such that $C^*_r(G)$ is nuclear are amenable \cite{AD02}, \cite{Ng}. Almost connected groups are amenable at infinity \cite[Theorem 6.8]{KW99bis}, \cite[Proposition 3.3]{AD02}. Their full $C^*$-algebras are nuclear. So they are  in $[InnAmen]$  or in $[Tr]$  if and only if they are amenable. This latter observation applies also to groups of type I.

We denote by $[\cC]$ the class of locally compact groups for which $C^*$-exactness is equivalent to KW-exactness. It contains  $[InnAmen]$ and $[Tr]$. Almost connected groups are in $[\cC]$ since they are KW-exact. The case of  groups of type I is not clear. Of course they are $C^*$-exact  but are they KW-exact?  In support of this question we point out that it is conjectured that every second countable locally compact group of type I has a cocompact amenable subgroup \cite{CKM}, a property which implies amenability at infinity \cite[Proposition 5.2.5]{AD-R}.  

It would be interesting to find more examples in the class $[\cC]$. It seems difficult to find examples not in $[\cC]$. Note that this class is  preserved  by extensions 
$$0\to N\to G\to G/N\to 0$$
 where $N$ is amenable.
Indeed since $N$ is amenable, $C^*_{r}(G/N)$ is a quotient of $C^*_{r}(G)$ (see the proof of Lemma 3.5 in \cite{CZ}). Assume that  $G$ is $C^*$-exact and that $G/N\in [\cC]$. Then the group $G/N$ is  $C^*$-exact and therefore KW-exact. It follows that $G$ is KW-exact since the class of KW-exact groups is preserved under extension \cite[Proposition 5.1]{KW99bis}.
\vspace{2mm}

\item[(4)] There are examples of non-inner amenable groups in $[Tr]$ (see \cite[Remark 2.6 (ii)]{FSW}, \cite[Example 4.15]{Man}). But are there inner amenable groups which are not in $[Tr]$? Note that the subclass $[IN]$ of $[InnAmen]$ is contained into $[Tr]$ \cite[Theorem 2.1]{FSW}. Let us recall that a locally compact group $G$ is in $[IN]$ if its identity $e$ has a compact neighborhood invariant under conjugacy. By \cite[Proposition 4.2]{Tay}, this is equivalent to the existence of a normal tracial state on the von Neumann algebra $L(G)$ of $G$. Since $C^*_{r}(G)$ is weakly dense into $L(G)$ the conclusion follows.

Let us observe that the existence of a locally compact group in $[InnAmen]\setminus [Tr]$ is equivalent to the existence  of a totally disconnected locally compact group in $[InnAmen]\setminus [Tr]$. Indeed let $G$ be a locally compact group in this set. Let $G_0$ be the connected component of the identity. Then $G_0$ is inner amenable as well as $G/G_0$ (see \cite[Corollary 3.3]{CT} and \cite[Proposition 6.2]{LP}). Since a connected inner amenable group is amenable (by \cite[Theorem 5.8]{AD02}), we see that $G_0$ is amenable. It follows that $C^*_{r}(G/G_0)$ is a quotient of $C^*_{r}(G)$ and therefore $G/G_0$ is not in $[Tr]$. 

As a consequence of this observation we are left with the following problem: does there exist a totally disconnected locally compact inner amenable group without open normal amenable subgroups?
\end{itemize}

\subsection{About exactness for groupoids}

\begin{itemize}

\item[(5)] In \cite{KW95}, Kirchberg and Wassermann have constructed examples of continuous fields of exact $C^*$-algebras on a locally compact space, whose $C^*$-algebra of continuous sections vanishing at infinity is not exact.  Are there examples  of \'etale group bundle groupoids $\cG$, whose reduced $C^*$-algebra is not exact whereas 
$$(C^*_{r}(\cG), \set{\pi_x : C^*_{r}(\cG) \to C^*_{r}(\cG(x))}_{x\in \cG^{(0)}},\cG^{(0)})$$
 is a continuous field of exact $C^*$-algebras (compare with Proposition \ref{prop:HLS})?
 
 A similar  question is asked in \cite[Question 3]{Lal17}: if $\cG$ is an inner exact  locally compact group bundle groupoid, whose fibres are  KW-exact groups, is it true that $\cG$ is  KW-exact?
\vspace{1mm}

\item[(6)]Let $\cG$ be a locally compact groupoid. We have
\begin{center}
Amenability at infinity $\xLongrightarrow{\text{(1)}}$ KW-exactness $\xLongrightarrow{\text{(2)}}$ $C^*$-exactness.
\end{center}
Let us recap what is known about the reversed arrows and what is still open.

 When $\cG$ is an \'etale inner amenable groupoids, the above three notions of exactness are equivalent (Theorem \ref{thm:equiv}). Does this fact extend to any inner amenable locally compact groupoid? Without the assumption of inner amenability, nothing is known.

As already said, it seems difficult to find an example of a locally compact group which is $C^*$-exact but not KW-exact. Could it be easier to find an example in the context of locally compact groupoids?

If $G$ is a KW-exact locally compact group, every semidirect product groupoid (relative to  a global or to a partial action) is strongly amenable at infinity \cite[Proposition 4.3, Proposition 4.23]{AD16}, and therefore KW-exact. Does KW-exactness of $X\rtimes G$ imply that $X\rtimes G$ is amenable at infinity in general?  Note that the notion of exactness for a semidirect product groupoid $X\rtimes \Gamma$, where $\Gamma$ is a discrete group is not ambiguous by Theorem \ref{thm:equiv}.

\end{itemize}

\subsection{(WCP) vs amenability}
\begin{itemize}
\item[(7)] Are there examples of inner exact groupoids $\cG$ which have the (WCP) without being amenable? By \cite{Bon} one should look for examples for which the orbit space $\cG\setminus \cG^{(0)}$  is not $T_0$. 

\vspace{2mm}

\item[(8)]   We have seen that if an \'etale  locally compact groupoid $\cG$ is  assumed to be strongly amenable at infinity, the (WCP)  implies its amenability (Theorem \ref{Kranz}).    Is it true in general, or at least  for  a semidirect product groupoid  $X\rtimes G$ with $X$ and $G$ locally compact?   Recall that this holds when  $G$ is KW-exact \cite[Theorem 5.15]{BEW20}, a property stronger than strong amenability at infinity.
\vspace{2mm}

\item[(9)]   Let $G$ be a discrete group and let $X = \partial G = \beta G\setminus G$ be its boundary equipped with  the natural action of $G$.  The weak containment property for $\partial G \rtimes G$ implies that this groupoid is amenable. Indeed the (WCP) implies that the sequence
 $$0\longrightarrow C^*_{r}(G\rtimes G) \longrightarrow  C^*_{r}( \beta G\rtimes  G)\longrightarrow  C^*_{r}(\partial G\rtimes G)\longrightarrow 0$$
is exact. Roe and Willett have proven in \cite{RW} that this exactness property implies that $G$  is exact. It follows that $G\actson \beta G$ is amenable and therefore  $G\actson \partial G$ is amenable too. Can we replace $\partial G$ by $\beta G$, that is, if $G\actson \beta G$ has the (WCP) can we deduce that $G\actson \beta G$ is amenable? This is asked in \cite[Remark 4.10]{BEW20bis}.
\end{itemize}
\vspace{0.5cm}

\noindent{\em Acknowledgements}. I thank Julian Kranz, Jean Renault and the referee for useful remarks and suggestions.

\vspace{2mm}

\noindent {\bf Addendum.}  A construction due to Suzuki \cite{Suz} gives an example of a totally disconnected locally compact inner amenable group without open normal amenable subgroups (or equivalently without tracial states), thus answering the question posed in point 7.2.(4) above.

Suzuki considers a sequence $(\Gamma_n, F_n)$ of pairs of discrete groups, where $F_n$ is a finite group acting on $\Gamma_n$ is such a way that the reduced $C^*$-algebra of the semidirect product $\Gamma_n \rtimes F_n$ is simple. Let us set $\Gamma = \bigoplus_{i=1}^{\infty}\Gamma_i$ and $K= \prod_{i=1}^\infty F_i$. Let  $K$ act  on $\Gamma$ component-wise. Then Suzuki shows that $C^*_{r}(G)$ is simple and that the Plancherel weight is the unique lsc semifinite  tracial weight on $C^*_{r}(G)$. Since  $G$ is not discrete this weight is not finite and therefore $G$ in not in $[Tr]$ (see also \cite[Remark 2.5]{FSW}).

Let us show now  that $G\in [InnAmen]$. We set $G_n = (\bigoplus_{i=1}^n \Gamma_i)\rtimes K$. Then $(G_n)$ is an increasing sequence of open subgroups of $G$ with $\bigcup_{n=1}^\infty G_n = G$. Since $G_n$ contains an open compact normal subgroup, namely $\prod_{i= n+1}^\infty F_i$ we see that there exists an inner invariant mean on $L^\infty(G_n)$ and therefore a mean $m_n$ on $L^\infty(G)$ which is invariant by conjugation under the elements of $G_n$.  Any cluster point of the sequence $(m_n)$ in $L^\infty(G)^*$ gives an inner invariant mean on $L^\infty(G)$.

\bibliographystyle{plain}

\end{document}